\documentclass[a4paper,11pt]{article}
\usepackage{authblk}
\usepackage[english]{babel} 
\usepackage[centertags,intlimits]{amsmath}
\usepackage{amsfonts}
\usepackage{amsthm}
\usepackage{amssymb}
\usepackage{pgf, tikz, pgfplots, float, graphicx, tikz-cd}
\usetikzlibrary{decorations.pathreplacing,calligraphy, arrows}
\usepackage{comment}
\usepackage{cite}

\newtheorem{theorem}{Theorem}[section]
\newtheorem{proposition}[theorem]%
{Proposition}

\newtheorem{lemma}[theorem]%
{Lemma}
\newtheorem{corollary}[theorem]%
{Corollary}

\newtheorem{example}{Example}[section]  
\numberwithin{equation}{theorem}

\begin{document}

\title{Proximity and Radius in Outerplanar Graphs with Bounded Faces}

\author[1]{Peter Dankelmann\thanks{Financial support by the South African National Research Foundation is greatly acknowledged.}}
\author[1,2]{Sonwabile Mafunda}
\author[1]{Sufiyan Mallu\thanks{The results of this paper form part of third author's PhD thesis.}
\thanks{Financial support by the South African National Research Foundation is greatly acknowledged.}}
\affil[1]{University of Johannesburg\\
South Africa}
\affil[2]{Soka University of America\\
USA}

\maketitle

\begin{abstract}
Let $G$ be a finite, connected graph and $v$ a vertex of $G$. 
The average distance and the eccentricity of $v$ in $G$ are defined as the arithmetic 
mean and the maximum, respectively, of the distances from $v$ to all other vertices of $G$. 
The proximity of $G$ and the radius of $G$ are defined as the minimum of the average 
distances and the eccentricities over all vertices of $G$. 

In this paper, we establish an upper bound on the proximity of a $2$-connected outerplanar graphs 
in terms of order and maximum face length. This bound is sharp apart from a small additive constant.  

It is known that the radius of a maximal outerplanar graph is at most $\lfloor \frac{n}{4} \rfloor +1$.  In the second part of this paper we show that this bound on the radius holds for a much 
larger subclass of outerplanar graphs, for all  
$2$-connected outerplanar graphs of order $n$ whose maximum face length does not exceed 
$\frac{n+2}{4}$.
\end{abstract}
MSC-class: 05C12

\section{Introduction}
Let $G$ be a finite, connected graph of order $n \geq 2$ with vertex set $V(G)$. 
The \textit{average distance} $\overline{\sigma}(v, G)$ and the \textit{eccentricity} 
${\rm ecc}(v,G)$ of a vertex $v\in V(G)$ are defined as the arithmetic mean and the maximum, 
respectively,  of the distances from $v$ to all other vertices of $G$, i.e.
$\overline{\sigma}(v, G)=\frac{1}{n-1} \sum_{u \in V(G)} d(v,u)$,
${\rm ecc}(v, G)=\max_{u \in V(G)} d(v,u)$,
where $d(v,u)$ denotes the usual shortest path distance between vertices $v$ and $u$.  
The {\em proximity} $\pi(G)$ of $G$ is defined as $\min_{v\in V(G)} \overline{\sigma}(v,G)$.  
The \textit{radius} ${\rm rad}(G)$ of  $G$ is $\min_{v\in V(G)} {\rm ecc}(v,G)$.\\
We note that the proximity is closely related to the minimum status of a graph, which is defined
as $(n-1)\pi(G)$.

A bound on proximity and remoteness (defined as the maximum of the average distances of the 
vertices, denoted by $\rho(G)$), in terms of order was determined by Zelinka \cite{Zel1968} and later, independently, by Aouchiche and Hansen \cite{AouHan2011}, who
introduced the names proximity and remoteness.  

\begin{theorem}
{\rm (Zelinka \cite{Zel1968}, Aouchiche, Hansen \cite{AouHan2011})} \\
\label{theo:bounds-on-proximity-given-order}
Let $G$ be a connected graph of order $n\geq 2$. Then
\[ 1\leq\pi(G) \leq  \left\{ \begin{array}{cc}
\frac{n+1}{4} & \textrm{if $n$ is odd,} \\
\frac{n+1}{4} + \frac{1}{4(n-1)} & \textrm{if $n$ is even.} 
\end{array} \right. \]
The lower bound holds with equality if and only if $G$ has a vertex of degree $n-1$. The upper bound holds with equality if and only if $G$ is a path or a cycle. Also,
\[ 1\leq\rho(G) \leq \frac{n}{2}. \]
The lower bound holds with equality if and only if $G$ is a complete graph. The upper bound holds with equality if and only if $G$ is a path.
\end{theorem}

The following bound on radius is folklore. 

\begin{proposition} \label{prop:bound-on-radius-given-order}
Let $G$ be a connected graph of order $n$. Then 
\[ {\rm rad}(G) \leq \frac{n}{2}. \]
\end{proposition}

Since by Theorem \ref{theo:bounds-on-proximity-given-order} the graphs that maximise the
proximity among graphs of given order have only vertices of small degree, it is natural
to expect improved bounds if also vertex degrees are taken into account. 
It was shown by Dankelmann \cite{Dan2015} that for graphs of minimum degree $\delta \geq 3$, the bound 
in Theorem \ref{theo:bounds-on-proximity-given-order} can be improved by a factor of 
about $\frac{3}{\delta+1}$. For further results on proximity involving minimum
degree see \cite{Dan2016, DanJonMaf2021, DanMaf2022}. 
Rissner and Burkhard \cite{RisBur2014} determined trees that minimise and maximise the proximity 
among all trees of given order and maximum degree. For a 
bound that takes into account both, maximum and minimum degree, see \cite{DanMafMal2022}. 
Bounds on the proximity of trees in terms of order and among others, number of vertices of odd degree 
were given by Peng and Zhou \cite{PenZho2021}.  
Guo and Zhou \cite{GuoZho2023} determined the trees that minimise the proximity
among all trees with a given degree sequence. 

Several relations between proximity and other distance based graph parameters 
were proved by Aouchiche and Hansen \cite{AouHan2011}. Hua and Das \cite{HuaDas2014} 
resolved conjectures relating proximity or remoteness to average distance of the graph
(defined as the arithmetic mean of  the distances between all pairs of vertices), maximum degree and average degree. They also found relations between proximity, remoteness and clique number 
(defined as the order of a largest complete subgraph).   
The maximum value of the difference between 
average eccentricity (defined as the arithmetic mean of the eccentricities of the vertices
of $G$) and proximity was determined by Ma, Wu and Zhang \cite{MaWuZha2012}.
Pei et al. \cite{PeiPanWanTia2021} determined the maximum difference between domination number (defined as the 
minimum cardinality of a set of vertices of $G$ so that every vertex not in the set is adjacent 
to some vertex in the set) and proximity.
The differences between average distance and proximity and between average eccentricity and 
remoteness were considered by Sedlar \cite{Sed2013}.
For relations between proximity, remoteness and girth (defined as the length of a shortest cycle) see  \cite{AouHan2017}.

Also proximity in various graph classes has been explored. 
Barefoot, Entringer and Szek\'{e}ly \cite{BarEntSze1997} obtained results on proximity and 
remoteness in trees. 
Several results for trees with no vertex of degree $2$ were obtained by 
Chen, Lin and Zhou \cite{CheLinZho2021}. 
Best possible bounds on proximity and remoteness in maximal planar graphs and related graphs 
were proved by Czabarka et al.\ in \cite{CzaDanOlsSze2021, CzaDanOlsSze2022},
see also \cite{DanMafMal-manu}. 
Recently, the bound in Theorem \ref{theo:bounds-on-proximity-given-order} was extended to digraphs by 
Ai, Gerke, Gutin and Mafunda \cite{AiGerGutMaf2021}.
Further results can be found in the survey on proximity and remoteness in graphs in \cite{AouRat2024}.

The radius of graphs has been studied extensively in the literature. The bound in 
Proposition \ref{prop:bound-on-radius-given-order} has been improved by a factor of about
$\frac{3}{\delta+1}$ for graphs of minimum degree $\delta$ in \cite{ErdPacPolTuz1989} 
(see also \cite{KimRhoSonHwa2012, Muk2014}). We only mention here the following bounds on the radius
pertinent to the present paper. It is well-known that the diameter
of a connected graph cannot exceed twice the radius. It was shown in \cite{ChaNem1984} that for 
chordal graphs, i.e., graphs with no induced cycle of length greater than $3$, the diameter is
always close to twice the radius. Bounds on the radius of maximal planar graphs and planar 
graphs with bounded face length were given in \cite{AliDanMuk2012}.

The starting point for this paper is the observation that the bounds on proximity in 
Theorem \ref{theo:bounds-on-proximity-given-order} and radius in 
Proposition \ref{prop:bound-on-radius-given-order} can be improved significantly 
(to about $\frac{n}{12}$ for proximity \cite{CzaDanOlsSze2022} and to about 
$\frac{n}{6}$ for the radius  \cite{AliDanMuk2012}) for maximal planar graphs. 
The latter bound was generalised to planar graphs with given maximum face length. 
The goal of this paper is to prove corresponding bounds on proxmitiy and radius for 
outerplanar graphs with bounded face lengths. We present an upper bound on 
proximity for $2$-connected outerplanar
graphs in terms of order and maximum face length, which is sharp apart from a small additive 
constant. 
It is fairly straightforward to prove that the radius of a maximal outerplanar graph of order
$n$ cannot exceed $\frac{n}{4}+1$. We show that the condition that the graph is maximal outerplanar,
i.e., that every face has length $3$, can be relaxed considerably by  proving that the same
conclusion holds for $2$-connected outerplanar graphs in which no face has length greater
than $\frac{n+2}{4}$.

This paper is organised as follows. 
In Section \ref{section:terminolory} we introduce the terminology and notation used in
this paper. 
In Section \ref{section:proximity-in-2-connected-OP}, using a novel approach, we derive an upper 
bound on the proximity of all 2-connected outerplanar graphs of given order and maximum face 
length, and we show that the bound is sharp apart from a small additive constant. 
Finally, in Section \ref{section:radius}, we present an upper bound on the radius of $2$-connected outerplanar graphs with bounded face lengths, and we show that this bound is sharp.

\section{Terminology and Notation}
\label{section:terminolory}

We use the following notation. 
Let $G=(V(G),E(G))$ be a finite, simple, connected graph, by $n(G)$ we 
denote the order of $G$, i.e., the number of vertices of $G$. 

The \emph{distance} $d_G(v,u)$ from vertex  $v$ to vertex $u$ in $G$ is the length of a shortest 
$(v,u)$-path in $G$. 
The \emph{transmission} of a vertex  $v\in V(G)$ is defined as 
$\sigma(v, G)=\sum_{u\in V(G)}d(u,v)$, so $\sigma(v,G) = (n(G)-1) \overline{\sigma}(v,G)$. 
If there is no danger of confusion, then we often drop the argument $G$. 
A vertex whose average distance equals $\pi(G)$ is called a \emph{median vertex} of $G$. 

The \emph{neighbourhood} of a vertex $v$ of $G$, denoted by $N(v)$, is the set of all vertices adjacent to $v$, and the cardinality $|N(v)|$ is the \emph{degree} of $v$, which we
denote by ${\rm deg}(v)$. 
A connected graph $G$ is \emph{$2$-connected} if deleting a vertex from $G$ leaves $G$ connected.

We say that a graph $G$ is \emph{outerplanar} if it can be embedded in the plane without any edge crossings such that every vertex of $G$ lies on the boundary of the exterior region.
The \emph{weak dual} of a such embedded outerplanar graph $G$, denoted by $G^{w}$, is the graph whose vertices correspond to the interior faces of $G$, with an edge between two vertices if and only if the corresponding faces in $G$ share a common edge. An outerplanar graph is said to be \emph{maximal outerplanar} if the addition of any new edge would render the graph no longer outerplanar. \\
In this paper we consider only maximal outerplanar graphs and $2$-connected outerplanar graphs.
Both are known to have a unique embedding in the plane. Hence we usually refer to the weak dual of 
the graph without specifying the embedding.  

For standard terminology on graphs not discussed here, we refer the reader to \cite{ChaLesZha2016}.

\section{Proximity in $2$-connected outerplanar graphs}
\label{section:proximity-in-2-connected-OP}

In this section we present an upper bound on the proximity of a $2$-connected outerplanar graphs of 
order $n$ and maximum face length $q$.   
We prove our bound by considering a well chosen weight 
function on the vertices of $G^w$. Then we demonstrate that the boundary of the face corresponding to a 
median vertex of the weighted weak dual $G^w$ contains a vertex of small average distance in $G$. 

We first recall the following two results due to Leydold and Stadler \cite{LeySta1998} 
and Fleischner et al. \cite{FleGelHar1974}.

\begin{theorem}{\rm (\cite{LeySta1998})}
\label{theo:Outerplanar-unique-cycle}
An outerplanar graph is Hamiltonian if and only if it is 2-connected. Furthermore, every 
$2$-connected outerplanar graph contains a unique Hamiltonian cycle.
\end{theorem}

\begin{proposition}{\rm (\cite{FleGelHar1974})}
\label{prop:dual-is-a-tree}
For any $2$-connected outerplanar graph $G$, its weak dual $G^w$ is a tree.
\end{proposition}

Assume that $G$ is a connected graph, and $c: V(G) \rightarrow \mathbb{R}^{\geq 0}$ is a 
nonnegative weight function on the vertices of $G$. We define the \textit{weighted total distance} 
of $G$ with respect to $c$ as
\[
\sigma_c(G) = \sum_{\{x,y\} \in V(G)} c(x)c(y) d_G(x,y).
\]
and for any vertex $v$ in $G$, the \textit{weighted distance} of $v$ with respect to $c$ is 
\[
\sigma_c(v,G) = \sum_{w \in V(G) \setminus \{v\}} c(w) d(v,w).
\]
A vertex $v$ that minimises $\sigma_c(v,G)$ among all vertices in $G$ is called 
a $c$-\textit{median vertex}. 
If $c$ is the weight function with $c(v)=1$ for all $v \in V(G)$, then the $c$-median vertices 
are the usual median vertices of $G$. 

Let \( T \) be a tree, and let \( c: V(T) \to \mathbb{R}^{\geq 0} \) be a nonnegative weight function on the vertices of \( T \). 
For a vertex \( v \) in \( T \), the \( c \)-branch weight \( \text{bw}_c(v, T) \) is the maximum weight of any component of \( T - v \).  If $X$ is a set of vertices, then we write $c(X)$ for $\sum_{x\in X} c(x)$.

\begin{proposition}{\rm (\cite{KarHak1979})} \label{prop:median-branchweight-weighted}  
Let $T$ be a tree, and let $c: V(G) \to \mathbb{R}^{\geq 0}$ be a nonnegative weight function on 
the vertices of $G$. A vertex $v$ of $G$ is a $c$-median vertex of $T$ if and only if 
${\rm bw}_c(v, T) \leq \frac{N}{2}$, where $N=c(V(T))$.
\end{proposition}

For the function $c$ that maps every vertex  to $1$, 
Proposition \ref{prop:median-branchweight-weighted} implies that $T$ has a vertex $v$ so that
no component of $T-v$ contains more than half the vertices of $T$. 
The following lemma proves that a similar statement holds for $2$-connected
outerplanar graphs.

\begin{lemma}
\label{lemm:face-component-has-at-most-n-2/2 vertices}
Let $G$ be a $2$-connected outerplane graph of order $n$. Then there exist a face, say $F$, such that every component of $G - V(F)$ has at most $\frac{n-2}{2}$ vertices.
\end{lemma}

\begin{proof}
Let $G$ be a $2$-connected outerplanar graph of order $n$. Denote the distinct interior faces of $G$ by 
$f_1, f_2, \ldots, f_t$ and their corresponding vertices in the weak dual $G^w$ by 
$f_1, f_2, \ldots, f_t$. 
By Proposition \ref{prop:dual-is-a-tree}, $G^{w}$ is a tree. 
Now we define a nonnegative weight function on the vertices of $G^w$ as follows, 
for all $i\in\{1,\ldots,t\} $, define $c(f_i)=\ell_{i}-2$, where $\ell_{i}$ indicates the length of face 
$f_i$. It is easy to prove by induction on the number of interior faces of $G$ that $c(V(G^w))=n-2$. 
Let $f_j$ be a $c$-median vertex of 
the weighted tree $G^w$. By Proposition \ref{prop:median-branchweight-weighted} we have 
${\rm bw}_c(f_i,G^w) \leq \frac{n-2}{2}$, so every component of $G - V(f_j)$ has at most 
$\frac{n-2}{2}$ vertices. Letting $f_j=F$, completes the proof of the lemma.
\end{proof}
 
 \begin{theorem}{\rm (\cite{Dan2012})}
\label{theo:total-weighted-distance-of-cycle-ch4}
Let $G$ be a cycle of order $n \geq 3$ with a nonnegative vertex weight function $c$ such that $c(G)=N$. Then
\[
\sigma_c(G) \leq
\begin{cases} 
\dfrac{nN^2}{8}  & \text{if } n \text{ is even}, \\[10pt]
\dfrac{(n^2-1)N^2}{8n}  & \text{if } n \text{ is odd}.
\end{cases}
\]
When $n$ is even, equality holds if and only if the weights of opposite vertices in $G$ are equal. When $n$ is odd, equality holds if and only if all vertices have the same weight.
\end{theorem}

The following corollary is an immediate consequence of the above Theorem \ref{theo:total-weighted-distance-of-cycle-ch4}.

\begin{corollary}
\label{coro:c-median-vertex-of-a-weighted-cycle}
Let $G$ be a cycle of order $n \geq 3$ with a nonnegative vertex weight function $c$ with $c(G)=N$. 
Then there exists a vertex $v$ such that 
\[
\sigma_c(v, G) \leq
\begin{cases} 
\dfrac{nN}{4}  & \text{if } n \text{ is even}, \\[10pt]
\dfrac{(n^2-1)N}{4n}  & \text{if } n \text{ is odd}.
\end{cases}
\]
\end{corollary}

\begin{proof} 
We only prove the corollary for even $n$; the proof for odd $n$ is analogous. 
Let \( G \) be a cycle of order \( n \), where $n$ is even, with a nonnegative vertex weight function 
\( c \), and let \( v \) be a \( c \)-median vertex. Suppose to the contrary that 
\( \sigma_c(v, G) > \dfrac{nN}{4} \). Since $\sigma_c(w,G) \geq \sigma_c(v,G)$ for all $w\in V(G)$, 
\[
\sigma_c(G) = \frac{1}{2}\sum_{w \in V(G)} c(w) \sigma_c(w, G) > \frac{1}{2}\sum_{w \in V(G)} c(w) \cdot \frac{nN}{4} = \frac{nN^2}{8}.
\]
However, the above inequality contradicts Theorem \ref{theo:total-weighted-distance-of-cycle-ch4}. 
Hence $\sigma_c(v, G)$ cannot be more than $\frac{nN}{4}$, this completes the proof of the corollary.
\end{proof}

We are now ready to establish an upper bound on the proximity of a $2$-connected outerplanar graph of order 
$n$ with bounded face length.

\begin{theorem}\label{theo:proximity in 2-connected outerplana graph}
If $G$ is a $2$-connected outerplane graph of order $n$, whose internal faces have lengths at most $q$, then 
\begin{equation*} \label{eq:pi(G) in terms of q,n}
\pi(G) \leq
   \dfrac{n+5}{8}+\dfrac{q^2-4q+9}{8(n-1)}.
\end{equation*}
\end{theorem}

\begin{figure}[H]
\begin{center}
\begin{tikzpicture}[
    scale=0.5,inner sep=0.8mm,
  vertex/.style={circle, thick, draw},
    edge/.style={draw, thick},
    dashededge/.style={draw, thick, dashed},
    face/.style={ellipse, draw, fill=red!20, minimum width=1.8cm, minimum height=2.8cm, inner sep=0pt},
    every label/.style={font=\large}
]

[
\node[vertex, label=above:$v_0$] (v0) at (-5,3) {};
\node[vertex, label=right:$v_1$] (v1) at (-2,2) {};
\node[vertex, label=right:$v_2$] (v2) at (-1,0) {};
\node[vertex, label=below:$v_3$] (v3) at (-2,-2) {};
\node[vertex, label=below:$v_i$] (v4) at (-5,-3) {};
\node[vertex, label=left:$v_{i+1}$] (v5) at (-8,-2) {};
\node[vertex, label=left:$v_{k-2}$] (v6) at (-9,0) {};
\node[vertex, label=left:$v_{k-1}$] (v7) at (-8,2) {};

\node[below] at (-5,0) {$F$} ;

\draw[edge] (v0) -- (v1);
\draw[edge] (v1) -- (v2);
\draw[edge,very thick, red] (v2) -- (v3);
\draw[dashededge, very thick, red] (v3) -- (v4);
\draw[edge] (v4) -- (v5);
\draw[dashededge, very thick, red] (v5) -- (v6);
\draw[edge, very thick, red] (v6) -- (v7);
\draw[edge] (v7) -- (v0);


\draw[-, very thick, red] (v0)..controls (-3,6.5) and (-2.5,6.5)..(v1); 

\node[above] at (-3.5,3.5) {$P_{0}$} ;

\draw[-, very thick, red] (v1) .. controls (3,4) and (1.5,1.5) .. (v2); 

\node[above] at (0,1.5) {$P_{1}$} ;
   
   \draw[-, very thick, red] (v4) .. controls (-8.5,-7.5) and (-8,-3.5) .. (v5); 
   
   \node[above] at (-6.6,-3.5) {$P_{i}$} ;
   
   \draw[edge] (-8,-3.5) -- (-5.27,-3.5);
   
    \draw[edge] (-7.9,-4.5) -- (-5.7,-4);

\draw[-, very thick, red] (v7) .. controls (-11,7) and (-6.5,5) .. (v0); 

  \draw[edge] (-8.4,2.5) -- (-6,4);
  
  \node[above] at (-8,3.5) {$P_{k-1}$} ;

\end{tikzpicture}{\vspace{-1.5cm}}
\caption{A 2-connected outerplanar graph $G$ with internal face $F$. The Hamiltonian cycle 
$C$ is indicated in red.}
\label{fig:outerplanar graph}
\end{center}
\end{figure}
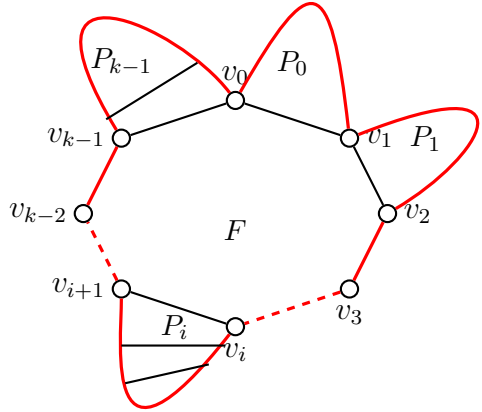

\begin{proof}
Let $G$ be a $2$-connected outerplane graph of order $n$, with face lengths bounded by $q$. 
By Lemma \ref{lemm:face-component-has-at-most-n-2/2 vertices}, $G$ contains a face $F$ with the property that every component of $G - V(F)$ has at most $\frac{n-2}{2}$ vertices.

Label the vertices of \( F \) as \( v_0, v_1, \ldots, v_{k-1} \) in clockwise order, such that 
\( v_iv_{i+1} \in E(G) \) for \( 1 \leq i \leq k-1 \), where indices are taken modulo $k$. 
By Theorem \ref{theo:Outerplanar-unique-cycle}, there exists a unique Hamiltonian cycle $C$. 
Let $C: u_1, u_2,\ldots,u_n, u_1$. Consider a segment $P: u_i, u_{i+1},\ldots,u_j$ of 
$C$ whose ends $u_i=v_i$ and $u_j=v_j$ are on the boundary of $F$, but its internal vertices $u_{i+1},\ldots,u_{j-1}$ are not. Then \( u_i \) and \( u_j \) must be consecutive on \( F \); otherwise, combining \( P \) with the subpath of \( F \) from \( u_i \) to \( u_j \) would form a cycle that encloses other vertices of \( G \). This would contradict the outerplanarity of \( G \). Hence we may assume that $C$ passes through the vertices of $F$ in the order $v_0, v_1,v_2, \ldots, v_{k-1}, v_0$. Denote the $(v_i, v_{i+1})$-segment of $C$ by $P_i$. Then $C$ equals $v_0,P_0,v_1,P_1,v_2,\ldots, P_{k-2},v_{k-1},P_{k-1},v_0$. 

Consider the spanning subgraph $H$ of $G$, whose edge set is $E(F) \cup \left[E(P_0)\cup E(P_1)\cdots \cup E(P_{k-1})\right]$. Clearly, the distance between two vertices in $G$ is at most their distance in $H$. Denote the number of internal vertices of $P_i$ by $p_i$. For an illustration see Figure \ref{fig:outerplanar graph}. 

Consider the graph \( H \) as a weighted graph where each vertex has a weight of 1. We now define a new weight function by reallocating each weight of a vertex to the nearest vertex in \( F \), where the weight of a vertex that has two nearest vertices in $F$ is split between these two vertices. More precisely, we define the weight function $c$ on the vertices of $F$ by 
\[ c(v_i) = 1 + \frac{1}{2} \left(p_i + p_{i-1}\right), \] where the indices are taken modulo $k$ and vertices not on $F$ have weight $0$. We now show that the total distance in $H$ of a vertex $w \in V(F)$ can be expressed as the sum of $\sigma_c(w,F)$ and a term that depends only on the 
$p_i$.

Fix a vertex $w \in V(F)$, $i \in \{0,1,\ldots,k-1\}$ and consider the sum of the distances between $w$ and the internal vertices of $P_i$. Let $P_i$ be the path $v_i, u_1, u_2,\ldots,u_{p_i},v_{i+1}$. First assume that $p_i$ is even. Then shortest paths between $w$ and $u_j$ can be chosen to pass through $v_i$ if $j \in \{1,2,\ldots,\frac{p_i}{2}\}$, and to pass through $v_{i+1}$ if $j \in \{\frac{p_i}{2}+1, \frac{p_i}{2}+2,\ldots,p_i\}$. Hence
\begin{align*}
\sum_{j=1}^{p_i} d_H(w,u_j)
 =& \; \sum_{j=1}^{\frac{p_i}{2}} [d_{H}(w,v_i) + d_{H}(v_i,u_j) ]
  + \sum_{j=\frac{p_i}{2}+1}^{p_i} [d_H(w,v_{i+1}) + d_H(v_{i+1},u_j) ] \\
  =& \; \frac{1}{2}p_i d_H(w,v_i) 
    + \frac{1}{2}p_i d_H(w,v_{i+1})
  + \frac{1}{4}p_i^2 + \frac{1}{2}p_i.
\end{align*}
Similarly, if $p_i$ is odd, then shortest paths between $w$ and $u_j$ can be chosen to pass through $v_i$ if $i \in \{1,2, \ldots, \frac{p_i-1}{2} \}$
and  through $v_{i+1}$ if $i \in \{ \frac{p_i+3}{2}, \frac{p_i+5}{2}, \ldots, p_i \}$, while for 
$j=\frac{p_i+1}{2}$ we have 
$d_H(w,v_j) \leq \frac{d_H(w,v_i) + d_H(w,v_{i+1})}{2} + \frac{p_{i+1}}{2}$. Hence we obtain

\[ 
 \sum_{j=1}^{p_i} d_H(w,u_j)
 \leq \frac{1}{2}p_i d_H(w,v_i) 
    + \frac{1}{2}p_i d_H(w,v_{i+1})
  + \frac{1}{4}(p_i+1)^2.\]  
Note that $\frac{1}{4}p_i^2 + \frac{1}{2}p_i < \frac{1}{4}(p_i+1)^2$. 
Summation over all $i \in \{0,1,\ldots,k-1\}$ yields
\begin{equation} \label{eq:sigma(w)-vs-sigma_c(w)}
\sigma(w,G) \leq \sigma(w,H) \leq \sigma_c(w,H) + \sum^{k-1}_{i=0} \left\lfloor\dfrac{1}{4}(p_i+1)^2 \right\rfloor. 
\end{equation}
For $x \in \mathbb{R}$ denote 
$\left\lfloor \frac{(x+1)^2}{4} \right\rfloor$ by $f(x)$. Since the weight of the vertices of $G$ is concentrated entirely in $F$, we have $\sigma_c(w,H) = \sigma_c(w,F)$, and thus 
\[ \sigma(w,G) \leq \sigma_c(w,F) 
     + \sum_{i=0}^{k-1} f(p_i). \]
Let $w$ be a $c$-median vertex of $F$. It suffices to show that 

\begin{equation} \label{eq:proximity of w in F}
\sigma(w,G) \leq  \frac{n^2+4n + k^2 - 4k +4}{8}.
\end{equation} 
{\sc Case 1:} $p_i \leq 
\lfloor \frac{n-k}{2} \rfloor $ 
for all $i \in \{0,1,\ldots,k-1\}$.\\[1mm]
Recall that $F$ is a cycle of length $k$, and the total weight of the vertices in $F$ is $n$. 
By Corollary \ref{coro:c-median-vertex-of-a-weighted-cycle} we have 
\begin{equation} \label{eq:bound-on-proximity-for-weighted-cycle- F}
\sigma_c(w, F) \leq
\begin{cases} 
\dfrac{kn}{4}  & \text{if } k \text{ is even}, \\[10pt]
\dfrac{(k^2-1)n}{4k}  & \text{if } k \text{ is odd}.
\end{cases}
\end{equation}
We now bound the transmission of $w$ in $G$. By combining \eqref{eq:sigma(w)-vs-sigma_c(w)} and 
\eqref{eq:bound-on-proximity-for-weighted-cycle- F} we obtain
\begin{equation} \label{eq:sigma(v_0)}
\sigma(w, G) \leq
\begin{cases} 
\dfrac{kn}{4} + \sum^{k-1}_{i=0} f(p_i) & \text{if } k \text{ is even}, \\[10pt]
\dfrac{(k^2-1)n}{4k} + \sum^{k-1}_{i=0} f(p_i) & \text{if } k \text{ is odd}.
\end{cases}
\end{equation}
Recall that $\sum_{i=0}^{k-1} p_i = n-k$, $p_i \in \mathbb{N}\cup \{0\}$ and 
$0 \leq p_i \leq \lfloor \frac{n-k}{2} \rfloor$ for all $i \in \{0,1,\ldots,k-1\}$. 
In order to maximise $\sum_{i=0}^{k-1}f(p_i)$ subject to these conditions,
we first observe the following inequality, which is easy to verify: If $a,b \in \mathbb{N}$ 
with $a \geq b \geq 1$ we have 
\begin{equation} \label{eq:maximising-sum-of-f(pi)} 
[ f(a+1) + f(b-1) ] - [ f(a) + f(b) ] 
    = \left\lceil \frac{a}{2} \right\rceil 
      - \left\lceil \frac{b-1}{2} \right\rceil
    \geq 0. 
\end{equation}    
Repeated application of \eqref{eq:maximising-sum-of-f(pi)} yields that $\sum f(p_i)$ is maximised subject to the above conditions if, in the case that $n-k$ is even,  
two of the $p_i$ equal $\frac{n-k}{2}$ and the remaining $p_i$ equal $0$ (if $n-k$ is even), and in the case that $n-k$ is odd, two of the $p_i$ equal $\frac{n-k-1}{2}$, one $p_i$ equals $1$, and the remaining $p_i$ equal $0$.
We thus obtain that
\[ \sum_{i-0}^{k-1} f(p_i)  \leq 
   \left\{ \begin{array}{cc}
  2 \lfloor \frac{1}{4}(\frac{n-k}{2}+1)^2 \rfloor & 
  \textrm{if $n-k$ is even,} \\
 2 \lfloor \frac{1}{4}(\frac{n-k-1}{2}+1)^2 \rfloor + 1 & 
  \textrm{if $n-k$ is odd,} 
  \end{array} \right\} 
\leq \frac{(n-k+2)^2}{8}. \]
In total we obtain for even $k$ that
\[ \sigma(w,G) \leq \frac{nk}{4} 
 + \frac{(n-k+2)^2}{8}= 
\frac{n^2+4n + k^2 - 4k +4}{8},
\]
and for odd $k$, 
\begin{align*}
\sigma(w,G) \leq \frac{(k^2-1)n}{4k} 
 + \frac{(n-k+2)^2}{8}=&\; 
\frac{n^2+4n + k^2 - 4k +4}{8}-\frac{n}{4k}\\
<&\;\frac{n^2+4n + k^2 - 4k +4}{8}.
\end{align*}
Thus (\ref{eq:proximity of w in F}) follows in Case 1.\\[1mm]
{\sc Case 2:} $p_i \geq \frac{n-k+1}{2}$ for some $i \in \{0,1,\ldots,k-1\}$. \\
We may assume without loss of generality that $p_0 \geq p_i$ for all $i \in \{1,2,\ldots,k-1\}$, 
and so $p_0 \geq \frac{n-k+1}{2}$.  Let $w=v_0$. 
We now bound $\sigma(w,G)$ in terms of $n$, $k$ and $p_0$. 
First consider $\sigma_c(w,F)$. 
Define $c'(v_i)=c(v_i)-1$ for all $i \in \{0,1,\ldots,k-1\}$. Then 
$\sigma_c(w,F) = \sigma(w,F) + \sigma_{c'}(w,F) = \left\lfloor \frac{k^2}{4} \right\rfloor + \sigma_{c'}(w,F)$. 
Clearly, $c'(w), c'(v_1) \geq \frac{p_0}{2}$, $c'(v_i) \geq 0$ for $i=2,3,\ldots,k-1$, and 
$\sum_{i=2}^{k-1} c'(v_i) = n-k-c'(w) - c'(v_1)$. Since every vertex of $F$ has distance at most 
$\lfloor \frac{k}{2} \rfloor$ from $w$, we have
\begin{align*}
\sigma_c(w, F) =&\;  \left\lfloor \frac{k^2}{4} \right\rfloor 
        +  1 \cdot c'(v_1) +\sum_{i=2}^{k-1} d(w,v_i) c'(v_i) \\
         \leq&\; \left\lfloor \frac{k^2}{4} \right\rfloor + c'(v_1) + \left\lfloor \frac{k}{2} \right\rfloor (n-k-c'(w) -c'(v_1)).
\end{align*}
Since $c'(w), c'(v_1) \geq \frac{1}{2}p_0$ we obtain 
\begin{equation} \label{eq:outerplanar-case2-1} 
\sigma_c(w, F) \leq  \left\lfloor \frac{k^2}{4} \right\rfloor + \frac{1}{2}p_0 + (n-k-p_0)  \left\lfloor \frac{k}{2} \right\rfloor. 
\end{equation}           
Now consider $\sum_{i=0}^{k-1} f(p_i)$. 
Recall that $p_i \in \mathbb{N}\cup \{0\}$ for $i=0,1,\ldots,k-1$ and $\sum_{i=0}^{k-1} p_i = n-k$. 
Assume that $p_0$ is fixed.  
Repeated application of \eqref{eq:maximising-sum-of-f(pi)} 
yields that $\sum_{i=0}^{k-1} f(p_i)$ is maximised subject to these conditions if there exists
$j\in \{1,2,\ldots,k-1\}$ for which $p_j=n-k-p_0$, and the remaining $p_i$ equal $0$. Hence
\begin{equation} \label{eq:outerplanar-case2-2}
 \sum_{i=0}^{k-1} f(p_i) \leq f(p_0) + f(n-k-p_0)
    \leq \frac{(p_0+1)^2}{4} + \frac{(n-k-p_0+1)^2}{4}. 
\end{equation}    
Combining \eqref{eq:outerplanar-case2-1} and \eqref{eq:outerplanar-case2-2}, we get
\begin{equation}
\sigma_c(w,G) \leq  \left\lfloor \frac{k^2}{4} \right\rfloor 
           + \frac{1}{2}p_0 + (n-k-p_0)  \left\lfloor \frac{k}{2} \right\rfloor
           + \frac{(p_0+1)^2}{4} + \frac{(n-k-p_0+1)^2}{4}.
\end{equation}
The derivative of the right hand side with respect to $p_0$ equals 
$\frac{1}{2}k - \lfloor \frac{1}{2}k \rfloor - \frac{1}{2}n + p_0 + \frac{1}{2}$, which by 
$p_0 \leq \lfloor \frac{n-2}{2} \rfloor$ is nonpositive, hence the right hand side is 
decreasing in $p_0$. Substituting $p_0=\frac{n-k+1}{2}$ yields, after simplification, 
\[ \sigma_c(w,G) \leq \frac{1}{8}(n^2 + 2n + k^2 - 6k+ \frac{1}{2}) < \frac{n^2+4n + k^2 - 4k +4}{8},\]
and so (\ref{eq:proximity of w in F}) follows in Case 2.

The derivative of the right-hand side of inequality~\eqref{eq:proximity of w in F} with respect to 
$k$ is given by $\frac{2k - 4}{8}$, which is nonnegative since $k \geq 3$. Thus, we conclude that 
the right-hand side of inequality~\eqref{eq:proximity of w in F} is increasing in $k$. Substituting 
$k \leq q$ and dividing by $n - 1$ in inequality~\eqref{eq:proximity of w in F} yields the bound 
stated in the theorem, which completes the proof.
\end{proof}

We now show that the bound obtained in Theorem \ref{theo:proximity in 2-connected outerplana graph} is best possible apart from the value of the additive constant. We first consider the case that 
$q\geq 4$. 

\begin{example}{\rm
Given an integer $q \geq 4$. For every $n \in \mathbb{N}$ with $n\geq q$ and $n-q$ a multiple of $4$, 
we construct a $2$-connected outerplanar graph $H_{n,q}$ of order $n$ with maximal internal 
face $q$.  

First let $q$ be even. Let $P:a_0, a_1,\ldots,a_{\frac{n}{2}-1}$ and 
$Q:b_0, b_1,\ldots,b_{\frac{n}{2}-1}$ two paths on $\frac{n}{2}$ vertices each. Let $H_{n,q}$ be the graph obtained from the disjoint union of $P$ and $Q$ by adding the edge $a_i b_i$ for all 
$i \in \{0,1,\ldots,\frac{n-q}{4}\} \cup \{\frac{n+q}{4}-1, \frac{n+q}{4}, \ldots, 
   \frac{n}{2}-1\}$. 
Clearly, $H_{n,q}$ has order $n$, is $2$-connected and outerplanar, and the face containing the vertices $a_i$ and $b_i$ for 
$i=\frac{n-q}{4}, \frac{n-q}{4}+1,\ldots, \frac{n+q}{4}-1$ has length $q$, while all other faces have length $4$. 

It is easy to verify that the median vertices of $H_{n,q}$ are exactly the vertices on the 
boundary of the face of length $q$. Let $u$ be such a vertex. We may assume that $u$ is on $P$. Clearly, 
$d(u,a_i) + d(u,a_{\frac{n}{2}-1-i}) = d(a_i,a_{\frac{n}{2}-1-i})  = \frac{n}{2}-1-2i$ 	and 
$d(u,b_i) + d(u,b_{\frac{n}{2}-1-i}) = d(b_i,b_{\frac{n}{2}-1-i}) +2  = \frac{n}{2}+1-2i$ for 
$i \in \{0,1,\ldots, \frac{n-q}{4}-1\}$. 
On the boundery of the face of length $q$, vertex $u$ has two vertices at distance $i$ for 
$i=1,2,\ldots, \frac{q-1}{2}$. Hence we obtain 
\[  \sigma(u,H_{n,q}) = \sum_{i=0}^{\frac{n-q}{4}-1} \left( \frac{n}{2}-1-2i \right)  
              + \sum_{i=0}^{\frac{n-q}{4}-1} \left( \frac{n}{2}+1-2i \right)  
              + 2 \sum_{i=1}^{\frac{q-1}{2}} i, \]
which, after simplification and division by $n-1$, yields 
\[ \pi(H_{n,q}) =  \overline{\sigma}(u,H_{n,q}) 
  =  \dfrac{n+5}{8}+\dfrac{q^2-4q+5}{8(n-1)}.\]
Now let $q$ be odd. $P:a_0, a_1,\ldots,a_{\frac{n-1}{2}-1}$ and 
$Q:b_0, b_1,\ldots,b_{\frac{n+1}{2}-1}$ two paths on $\frac{n-1}{2}$ and $\frac{n+1}{2}$ vertices,
respectively.  Let $H_{n,q}$ be the graph obtained from the disjoint union of $P$ and $Q$ by 
adding the edge $a_i b_i$ for all $i \in \{0,1,\ldots,\frac{n-q}{4}\}$ 
and the edge $a_ib_{i+1}$ for all $i \in \{\frac{n+q-2}{4}-1, \frac{n+q-2}{4}-2, \ldots, 
   \frac{n-1}{2}-1\}$. 
Again, $H_{n,q}$ has order $n$, is $2$-connected and outerplanar, and has one face of length $q$, 
while all other faces have length $4$. Calculations similar to those for even $q$ yield that
\[ \pi(H_{n,q}) =  \dfrac{n+3}{8}+\dfrac{q^2-2q-1}{8(n-1)}.\]   
and so $\pi(H_{n,q})$ (for both even and odd value(s) of $q$) differs from the bound in 
Theorem \ref{theo:proximity in 2-connected outerplana graph}  by less than $\frac{1}{4}$. 
Hence the bound obtained in Theorem \ref{theo:proximity in 2-connected outerplana graph} is sharp apart from an additive constant.
}
\end{example}

In maximal outerplanar graphs, every face has length $3$. Hence we obtain 
from Theorem \ref{theo:proximity in 2-connected outerplana graph} for $q=3$
the following corollary. 

\begin{corollary}  \label{coro:proximity-in-mops}
Let $G$ be a maximal outerplanar graph of order $n$. Then 
\[ \pi(G) \leq  \frac{n+5}{8} + \frac{3}{4(n-1)}. \]
\end{corollary}

The following example shows that the bound obtained in Corollary \ref{coro:proximity-in-mops}
is sharp apart from an additive constant.

\begin{example}{\rm
Let $n \in \mathbb{N}$ with $n \geq 10$ be given. 
Define $k = \lfloor \frac{n+2}{4}\rfloor$ and $k' = n-4k+4$. Then $k' \in \{2, 3, 4, 5\}$. 
We define the graph $H_{n,3}$ as follows. Let $G_0$ be the graph consisting of a single vertex $a_0$. 
For each $i \in \{1, 2, \ldots, k-2\}$ let $G_i$ be a copy of the path $P_4$ with vertices $a_i, b_i, c_i, d_i$, 
where $a_i$ and $d_i$ are the end-vertices of the path. Let $G_{k-1}$ be a copy of the path $P_3$ with vertices 
$a_{k-1}, b_{k-1}, c_{k-1}$. Furthermore, let $G_k$ be a copy of the path $P_{k'}$ with vertices 
$x^1_k, x^2_k, \ldots, x^{k'}_k$. To the disjoint union of the graphs 
$G_0, G_1, \ldots, G_k$ we add edges as follows: 
an edge between $a_0$ and each of  $a_1, b_1, c_1,$ and $d_1$ is added,  
for each $i \in \{1, 2, \ldots, k-3\}$ the edges $a_i a_{i+1}$, $b_i b_{i+1}$, $c_i c_{i+1}$ and $d_i d_{i+1}$ are added, 
for each $i \in \{1, 2, \ldots, k-2\}$ the edges $a_i b_{i+1}$ are added,  
for each $i \in \{1, 2, \ldots, k-3\}$ the edges $d_i c_{i+1}$ are added,  
the edges $a_{k-2} a_{k-1}$, $b_{k-2} b_{k-1}$, $c_{k-2} c_{k-1}$ and $d_{k-2} c_{k-1}$ are added,  
and finally, the edges $a_{k-1} x^1_k$, $a_{k-1} x^2_k$, $b_{k-1} x^2_k$, $b_{k-1} x^3_k$, \ldots, $b_{k-1} x^{k'}_k$ are added. 
Now, by removing the edges \( b_i c_i \) for \( i \in \{2, 3, \ldots, k-1\} \), we obtain the graph 
$H_{n,3}$. 
See Figure \ref{fig:construction of H} for a sketch of the graph $H_{19,3}$.

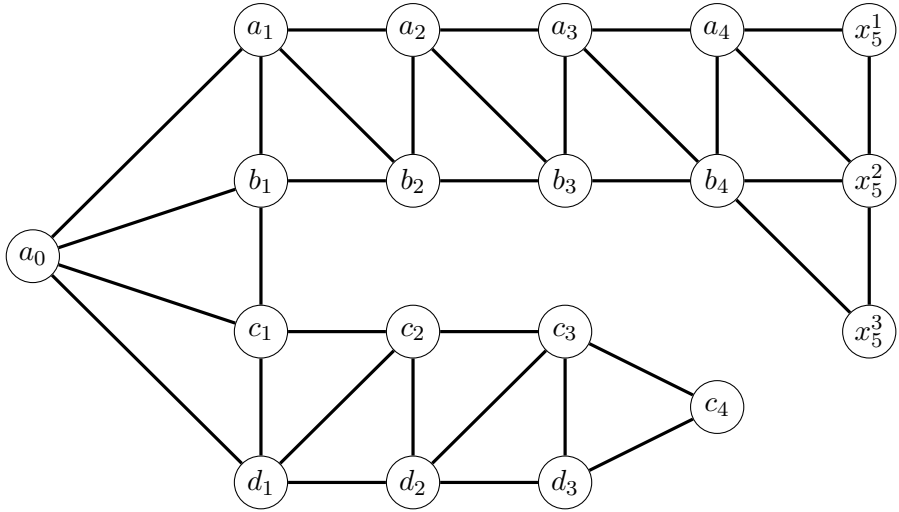
\begin{figure}[H]
	\begin{center}
\begin{tikzpicture}[every node/.style={circle, draw, minimum size=7mm, inner sep=0.5pt}]

\node (a0) at (0,3) {$a_0$};

\node (a1) at (3,6) {$a_1$};
\node (b1) at (3,4) {$b_1$};
\node (c1) at (3,2) {$c_1$};
\node (d1) at (3,0) {$d_1$};

\node (a2) at (5,6) {$a_2$};
\node (b2) at (5,4) {$b_2$};
\node (c2) at (5,2) {$c_2$};
\node (d2) at (5,0) {$d_2$};

\node (a3) at (7,6) {$a_3$};
\node (b3) at (7,4) {$b_3$};
\node (c3) at (7,2) {$c_3$};
\node (d3) at (7,0) {$d_3$};

\node (a4) at (9,6) {$a_4$};
\node (b4) at (9,4) {$b_4$};
\node (c4) at (9,1) {$c_4$};

\node (x1) at (11,6) {$x^1_5$};
\node (x2) at (11,4) {$x^2_5$};
\node (x3) at (11,2) {$x^3_5$};

\draw[very thick, black] (a0)--(a1)--(a2)--(a3)--(a4)--(x1);

\draw[very thick, black] (a0)--(b1)--(b2)--(b3)--(b4)--(x2);

\draw[very thick, black] (a0)--(c1)--(c2)--(c3)--(c4);

\draw[very thick, black] (a0)--(d1)--(d2)--(d3)--(c4);

\draw[very thick, black] (b4)--(x3);

\draw[very thick, black] (x1)--(x2)--(x3);

\draw[very thick, black] (a1)--(b1)--(c1)--(d1);

\draw[very thick, black] (a2)--(b2);
\draw[very thick, black] (a3)--(b3);
\draw[very thick, black] (a4)--(b4);

\draw[very thick, black] (c2)--(d2);
\draw[very thick, black] (c3)--(d3);

\draw[very thick, black] (a1)--(b2);
\draw[very thick, black] (a2)--(b3);
\draw[very thick, black] (a3)--(b4);
\draw[very thick, black] (a4)--(x2);

\draw[very thick, black] (d1)--(c2);
\draw[very thick, black] (d2)--(c3);

\end{tikzpicture}
\caption{The maximal outerplanar graph $H_{19,3}$ with $k=5$ and $k'=3$.}
	\label{fig:construction of H}
	\end{center}
	\end{figure}

Clearly, $H_{n,3}$ is a maximal outerplanar graph of order $n$. 
A tedious but straightforward calculation shows that $a_0$ is a median vertex of $H_{n,3}$,
and that

\[\pi(H_{n,3}) \; = \; \left\{
 	\begin{array}{ll}
 	\dfrac{n+5}{8}+\dfrac{11}{8(n-1)} \quad \qquad\mbox{if n $ \equiv 0 \pmod 4$,}\\
 	\\
 	\dfrac{n+5}{8}+\dfrac{1}{n-1} \qquad\qquad \mbox{if n $ \equiv 1 \pmod 4$,}\\
 	\\
 	\dfrac{n+5}{8}+\dfrac{11}{8(n-1)}  \quad\qquad\mbox{if n $ \equiv  2 \pmod 4$,}\\
 	\\
 	\dfrac{n+5}{8}+\dfrac{3}{2(n-1)} \quad\qquad \mbox{if n  $ \equiv  3 \pmod 4$.}\\
 	\end{array}
 	\right.\]
    }
\end{example}

The value $\pi(H_{n,3})$ differs from the bound in Corollary \ref{coro:proximity-in-mops} by a term
that approaches $0$ as $n$ gets large. Hence Corollary \ref{coro:proximity-in-mops} is close 
to being sharp. 

We note that in \cite{Mal-PhD} it was shown that the graph $H_{n,3}$ has maximum proximity among all
maximal outerplanar graphs of order $n$. \\

For completeness, we derive an upper bound on the remoteness of 
$2$-connected outerplanar graphs of given order

\begin{proposition}
Let $G$ be a $2$-connected outerplanar graph of order $n$. Then 
\[\rho(G) \leq \left\{  \begin{array}{cc}
\frac{n+1}{4} + \frac{1}{4(n-1)} & \textrm{if $n$ is even,} \\ 
\frac{n+1}{4}  & \textrm{if $n$ is odd,} 
     \end{array} \right.   \]
and this bound is sharp, even for maximal outerplanar graphs. 
\end{proposition}

\begin{proof}
Let $G$ be a $2$-connected outerplanar graph. 
By Theorem \ref{theo:Outerplanar-unique-cycle}, $G$ has a Hamiltonian cycle, i.e., a 
spanning subgraph isomorphic to $C_n$ . Hence 
$\rho(G) \leq \rho(C_n)$. A straightforward evaluation of $C_n$ shows that 
$\rho(C_n)$ equals $\frac{n+1}{4} + \frac{1}{4(n-1)}$ if $n$ is even,
and  $\frac{n+1}{4}$ if $n$ is odd. 

To see that the bound is sharp for every $n$, even for maximal outerplanar graphs,
consider the following graph. Given $n\in \mathbb{N}$ with $n\geq 3$. 
Let $G_n$ be the graph obtained from two paths
$a_0, a_1,\ldots,a_{\lfloor (n-1)/2 \rfloor}$ and 
$b_1, b_2,\ldots,b_{\lceil (n-1)/2 \rceil}$
by adding the edges $a_ib_j$ whenever $j \in \{i,i+1\}$. 
It is easy to verify that $G_n$ is a maximal outerplanar graph of order $n$
and that $\overline{\sigma}(a_0)$ equals $\frac{n+1}{4} + \frac{1}{4(n-1)}$ if $n$ is even,
and  $\frac{n+1}{4}$ if $n$ is odd. Hence $G_n$ attains the bound.
\end{proof}

\section{Radius in $2$-connected outerplanar graphs}
\label{section:radius}

In this section we present an upper bound on the radius of a $2$-connected outerplanar graphs with 
bounded face lengths. For maximal outerplanar graphs,  a sharp bound on 
radius is easily obtained from a result by Chang and Nemhauser \cite{ChaNem1984}
relating the radius and the diameter of a connected chordal graph (see below).  
Recall that a graph is said to be chordal if every induced cycle has length three.

\begin{theorem}{\rm (\cite{ChaNem1984})}\label{theoremcn1984}
Let $G$ be a connected chordal graph. Then 
\[2{\rm rad}(G)-2\leq {\rm diam}(G)\leq 2{\rm rad}(G).\]
\end{theorem}

\begin{corollary}
Let $G$ be a maximal outerplanar graph of order $n$. Then
\[{\rm rad}(G) \leq \big\lfloor \dfrac{n}{4} \big\rfloor +1.\]
\end{corollary}

\begin{proof}
Let $G$ be a maximal outerplanar graph. Then $G$ is chordal. Hence we have 
${\rm rad}(G) \leq \frac{1}{2}{\rm diam}(G) +1$ by Theorem \ref{theoremcn1984}. 
Also, since $G$ is $2$-connected, it follows that ${\rm diam}(G)\leq \frac{n}{2}$, 
and so  ${\rm rad}(G) \leq \dfrac{n}{4}+1$, as desired. 
\end{proof}

We will now show that this bound on the radius holds not only for maximal outerplanar graph, but holds in general for any $2$-connected outerplanar graphs in which the maximum length of an internal face does not exceed $\frac{n+2}{4}$.

\begin{theorem}\label{theo:rad in Outerpanar}
Let $G$ be a $2$-connected outerplane graph of order $n$ whose faces have length at most 
$\frac{n+2}{4}$. 
Then
\[ {\rm rad}(G) \leq \left\lfloor \frac{n}{4} \right\rfloor + 1. \]
\end{theorem}

\begin{proof}              
Let $F$, $k$, $v_i$, $P_i$ and $p_i$ be as defined in the proof of 
Theorem \ref{theo:proximity in 2-connected outerplana graph}.
As in the proof of Theorem \ref{theo:proximity in 2-connected outerplana graph} we may asssume 
that $p_0$ has the largest value 
among the $p_i$, and that $p_j$ has the largest value among the $p_i$ with $i \neq 0$. 
We may assume without loss of generality that $j \leq \frac{k}{2}$.

Since the vertices of $P_i$ together with $v_i$ and $v_{i+1}$ induce a Hamiltonian graph in 
$G$ that has order $p_i+2$,  every vertex of $P_i$ is within distance 
$\lfloor \frac{p_i+2}{2} \rfloor$ from both, $v_i$ and $v_{i+1}$. Define 
$ \ell := \lfloor \frac{n+4}{4} \rfloor - \lfloor \frac{p_0+2}{2} \rfloor+ 1$. 
Note that $\ell \geq 1$ since $p_0 \leq \frac{n-2}{4}$ by 
Lemma \ref{lemm:face-component-has-at-most-n-2/2 vertices}.
Let $u:=v_{\ell}$ if $\ell \leq j$, and $u:=v_j$ if $\ell >j$.
To prove the theorem it suffices to show that for all $x \in V(G)$, 
\begin{equation}
d_G(u, x) \leq \left\lfloor \frac{n+4}{4} \right\rfloor.
\end{equation}  
{\sc Case 1:} $x \in V(P_0)$. \\
Then $d_G(u,v_1) \leq \ell-1$ and $d_G(v_1,x) \leq \lfloor \frac{p_0+2}{2} \rfloor$. Hence
\[ d(u,x) \leq d_G(u,v_1) + d_G(v_1, x) 
\leq \left(\left\lfloor \frac{n+4}{4} \right\rfloor - \left\lfloor \frac{p_0+2}{2} \right\rfloor \right) + \left\lfloor \frac{p_0+2}{2} \right\rfloor
= \left\lfloor \frac{n+4}{4} \right\rfloor, \]
as desired.\\[1mm]
{\sc Case 2:}  $x \in V(P_j)$. \\
If $u=v_j$, then 
\[ d_G(u,x) = d_G(v_j,x) \leq \left\lfloor \frac{p_j+2}{2} \right\rfloor 
   \leq  \left\lfloor \frac{p_0+2}{2} \right\rfloor
   <  \left\lfloor \frac{n+4}{4} \right\rfloor, \]
with the last inequality holding since $p_0 \leq \frac{n-2}{2}$. 

If $u = v_{\ell}$, then, by $j \leq \lfloor \frac{k}{2}\rfloor$ and 
$d_G(v_j,x) \leq  \lfloor \frac{p_j+2}{2} \rfloor$ we have    
\[  d_G(u,x)  =  d_G(v_{\ell},v_j) + d_G(v_j,x) 
      \leq   \left\lfloor \frac{k}{2} \right\rfloor -\ell + \left\lfloor \frac{p_j+2}{2} \right\rfloor. \]
Hence
\begin{align*}
 d_G(u,x) \leq&\;  \left\lfloor \frac{k}{2} \right\rfloor -\left(\left\lfloor \frac{n+4}{4} \right\rfloor - \left\lfloor \frac{p_0+2}{2} \right\rfloor+ 1 \right)+ \left\lfloor \frac{p_j+2}{2} \right \rfloor\\
 \leq& \; \left\lfloor \frac{k+p_0 + p_j+2}{2} \right\rfloor  - \left\lfloor \frac{n+4}{4} \right\rfloor,
\end{align*}
Since $k+p_0+p_j \leq n$, we obtain
\[ d_G(u,x) \leq \left\lfloor \frac{n+2}{2} \right\rfloor - \left\lfloor \frac{n+4}{4} \right\rfloor
            \leq \left\lfloor \frac{n+4}{4} \right\rfloor, \]
as desired. \\[1mm]
{\sc Case 3:} $x \in V(P_i)$ for some $i \in \{1,\ldots,k-1\} - \{0,j\}$. \\
Since $u$ and $v_i$ are on the cycle $F$, we have $d_G(u,v_i) \leq \frac{k}{2}$. 
Moreover, $d_G(v_i,x) \leq \frac{p_i+2}{2}$. Hence 
\begin{equation} \label{eq:radius-of-outerplanar-1} 
d_G(u,x) \leq d_G(u,v_i) + d_G(v_i,x) \leq \frac{k}{2} +\frac{p_i+2}{2}. 
\end{equation}
We now consider two sub-cases. \\[1mm]
{\sc Case 3.1:} $p_i \geq k$. \\
We claim that $p_i \leq \frac{1}{2}(n-k-p_0)$. Indeed, since $k+p_0+p_j+p_i \leq n$,
we have $p_i + p_j \leq n-k-p_0$ and thus, since $p_i \leq p_j$ by the definition of $p_j$, 
the inequality $p_i \leq \frac{1}{2}(n-k-p_0)$ follows.
Substituting this into \eqref{eq:radius-of-outerplanar-1} yields 
\[ d_G(u,x) \leq \frac{n+k-p_0+4}{4}, \]
and since our asumtion $p_i \geq k$ implies that $p_0 \geq k$, the theorem follows in this case. \\[1mm]
{\sc Case 3.2:} $p_i \leq k-1$. \\
Then \eqref{eq:radius-of-outerplanar-1} in conjunction with $k \leq \frac{n+2}{4}$ yields that 
\[ d_G(u,x) \leq 2k \leq \frac{n+4}{4}, \]
and the theorem holds in this case. \\[1mm]
{\sc Case 4:} $x \in V(F)$. \\
Then $d_G(u,x) \leq \frac{k}{2}$. Since $k \leq \frac{n+2}{4}$, we have
$d_G(u,x) \leq \frac{n+2}{8} < \lfloor \frac{n}{4} \rfloor +1$, and the theorem follows.
\end{proof}

The following example shows that the bound in Theorem \ref{theo:rad in Outerpanar} is sharp 
for every $n$.

\begin{example}{\rm 
Let $n \in \mathbb{N}$ with $n\geq 4$. Define the graph $G_n$ as follows: 
 
If $n$ is even, then let $G_n$ be obtained from 
two disjoint paths $a_1,a_2,\ldots,a_{\frac{n}{2}}$ and $b_1, b_2,\ldots,b_{\frac{n}{2}}$ by adding the edges $a_1b_1, a_2b_2,\ldots,a_{\frac{n}{2}}b_{\frac{n}{2}}$.

If $n$ is odd, then let $G_n$ be obtained from $G_{n-1}$ by adding a new vertex adjacent to $a_1$ and $b_1$. 

Clearly, $G_n$ is $2$-connected and outerplanar, and every internal face has length $3$ or $4$. It is easy to verify that ${\rm rad}(G_n) = \lfloor \frac{n}{4} \rfloor +1$.}
\end{example}

We do not know if in Theorem \ref{theo:rad in Outerpanar}, the condition on the maximum
face length can be relaxed. For $n\in \mathbb{N}$ with $n\geq 3$ let $q_n$ denote the 
largest value so that the radius of every $2$-connected outerplanar graph with maximum 
interior face length not more than $q_n$ has radius at most $\frac{n}{4}+1$. 
Theorem \ref{theo:rad in Outerpanar} implies that $q_n \leq \frac{n+2}{4}$. On the other
hand $q_n < \frac{n}{2}+3$ since, as is easy to see, every outerplanar graph with an 
interior face of length $q$ has radius at least $\frac{q-1}{2}$.
It would be interesting to get better estimates for $q_n$.   



\begin{thebibliography}{99}
\addcontentsline{toc}{chapter}{\numberline{}Bibliography}

\bibitem{AiGerGutMaf2021} Ai, J.; Gerke, S.; Gutin, G.; Mafunda, S.; Proximity and remoteness in 
directed and undirected graphs. Discrete Math.\ {\bf 344} no.\ 3 (2021), 112252. 
     
\bibitem{AliDanMuk2012} Ali, P.; Dankelmann, P.; Mukwembi, S.; The radius of k-connected planar graphs with bounded faces. Discrete Appl. Math. {\bf 312} (2012), 3636–3642.
    
\bibitem{AouHan2011} Aouchiche, M.; Hansen, P.; Proximity and remoteness in graphs: results and conjectures. Networks {\bf 58} (no.\ 2) (2011), 95-102.
     
\bibitem{AouHan2017} Aouchiche, M.; Hansen, P.; Proximity, remoteness and girth in graphs. Discrete Appl.\ Math.\ {\bf 222} (2017), 31-39. 
         
\bibitem{AouRat2024} Aouchiche, M.; Rather, B.A.; Proximity and remoteness in graphs: a survey. Discrete Appl. Math. \ {\bf 353} (2024), 94-120.

\bibitem{BarEntSze1997} Barefoot, C.A.; Entringer, R.C.; Sz\'{e}kely, L.A.; Extremal values for ratios of distances in trees. Discrete Appl.\ Math.\ {\bf 80} (1997), 37-56.

\bibitem{ChaNem1984} Chang, G.J.; Nemhauser, G.L.; The $k$-domination and $k$-stability problems on graphs. SIAM J. Algebraic Discrete Meth. 5, (1984), 332-345.
   
\bibitem{ChaLesZha2016} Chartrand, G.; Lesniak, L.; Zhang, P.; Graphs \& Digraphs, CRC Press. $6^{th}$ Ed (2016).
         
\bibitem{CheLinZho2021} Cheng, M.; Lin, H.; Zhou, B.;  Minimum status of series-reduced trees with given parameters. Bull Braz.\ Math.\ Soc.\, New Series (2021), 1-20. https://doi.org/10.1007/s00574-021-00278-1 

 
\bibitem{CzaDanOlsSze2021} 
Czabarka, \'{E}.; Dankelmann, P.; Olsen,T.; Sz\'{e}kely, L.A.;  
Wiener index and remoteness in triangulations and quadrangulations. 
Discrete Math. Theor. Comp. Sci. {\bf 23}, no. 1 (2021). 
     
\bibitem{CzaDanOlsSze2022} Czabarka, \'{E}.; Dankelmann, P.; Olsen,T.; Sz\'{e}kely, L.A.; 
 Proximity in triangulations and quadrangulations. 
Electron. J.\ Graph Theory Appl.\ (EJGTA)  {\bf 10} (2) (2022), 425-446.
       
\bibitem{Dan2012} Dankelmann, P; Average distance in weighted graphs. Discrete Math. {\bf 312} (2012), 12-20.
    
\bibitem{Dan2015} Dankelmann, P.; Proximity, remoteness, and minimum degree. Discrete Appl.\ Math.\ {\bf 184} (2015), 223-228.  

\bibitem{Dan2016} Dankelmann, P.; New bounds on proximity and remoteness in graphs. 
Commun.\ Comb.\ Optim.\  {\bf 1} (2016) 28-40.     
    
\bibitem{DanJonMaf2021}  Dankelmann, P.; Jonck, E.; Mafunda, S.; Proximity and remoteness in triangle-free and $C_4$-free graphs in terms of order and minimum degree. Discrete Math.\ {\bf 344} no.\ 9 (2021), 112513.
      
\bibitem{DanMaf2022}  Dankelmann, P.; Mafunda, S.; On the difference between proximity and other distance parameters in triangle-free graphs and $C_4$-free graphs. Discrete Appl. Math.\ {\bf 321} (2022), 295-307.   

\bibitem{DanMafMal2022}  Dankelmann, P.; Mafunda, S.; Mallu, S.; 
Proximity, remoteness and maximum degree in graphs. 
Discrete Math. Theor. Comput. Sci. {\bf 24} no.\ 2 (2022).

\bibitem{DanMafMal-manu} Dankelmann, P.; Mafunda, S.; Mallu, S.; 
On proximity and other distance parameters in planar graphs (submitted).

\bibitem{ErdPacPolTuz1989} Erd\H{o}s, P.; Pach, J.;, Pollack, R.; Tuza. Z.; 
J.\ Combin.\ Theory B {\bf 47} (1989), 73-79. 
     
\bibitem{FleGelHar1974} Fleischner, H.J.; Geller, D.P.; Harary, F.; Outerplanar graphs and weak duals. J. Indian Math. Soc. {\bf 38} (1974), 215-219.

\bibitem{GuoZho2023} Guo, H.; Zhou, B.; 
Minimum status of trees with a given degree sequence.
Acta Inform. {\bf 60} (2023), 1-10.  
      
\bibitem{HuaDas2014}  Hua, H.;  Das, K.Ch.; Proof of conjectures on remoteness and proximity in graphs. Discrete Appl.\ Math.\ {\bf 171} (2014), 72-80.  

\bibitem{HuaCheDas2015} Hua, H.; Chen, Y.;  Das, K.Ch.; The difference between remoteness and radius of a graph. Discrete Appl.\ Math.\ {\bf 187} (2015), 103-110.   
   
   
\bibitem{KarHak1979} Kariv, O.; Hakimi, S.L.; An algorithmic approach to network location problems II: The p-medians. SIAM J.\ Appl.\ Math.\ {\bf 37} no.\ 3 (1979), 539-560. 

\bibitem{KimRhoSonHwa2012} Kim, B.M.: Rho, Y.: Song, B.C.:  Hwang, W.; 
   The maximum radius of graphs with given order and minimum degree. 
   Discrete Math. {\bf 312} no.\ 2 (2012), 207-212.
     
\bibitem{LeySta1998} Leydold, J.; Stadler, P.F.; Minimal Cycle Bases of Outerplanar Graphs. 
Electron. J. Combin. {\bf 5} (1998).

\bibitem{MaWuZha2012} Ma, B.; Wu, B.; Zhang, W.; Proximity and average eccentricity of a graph. Inform.\ Process.\ Lett.\ {\bf 112} (no.\ 10) (2012), 392-395. 
    
\bibitem{Mal2022} Mallu, S.; \textit{Proximity and remoteness in graphs.}  Master’s Thesis, University of  Johannesburg, (2022).

\bibitem{Mal-PhD} Mallu, S.; \textit{Bounds on proximity and remoteness in graphs and digraphs.}  
PhD Thesis, University of  Johannesburg, (2025).
     

\bibitem{Muk2014} Mukwembi, S.; 
On size, radius and minimum degree. 
Discrete Math.\ Theoret. Comp.\ Sci. {\bf 16}, (2014).  
     
\bibitem{PeiPanWanTia2021} Pei, L.; Pan, X.; Wang, K.; Tian, J.; Proofs of the AutoGraphiX conjectures on the domination number, average eccentricity and proximity. Discrete Appl. Math. {\bf 289} (2021), 292-301.   
    
\bibitem{PenZho2021} Peng, Z.; Zhou, B.; Minimum status of trees with given parameters. 
RAIRO Oper. Res. {\bf 55} (2021): S765-S785. 


\bibitem{RisBur2014} Rissner, R.; Burkard, R.E.; Bounds on the radius and status of graphs. Networks, {\bf 64} no.\ 2 (2014), 76-83.   

\bibitem{Sed2013} Sedlar, J.; Remoteness, proximity and few other distance invariants in graphs. Filomat {\bf 27} no.\ 8 (2013), 1425-1435.   

  
\bibitem{WuZha2014} Wu, B.; Zhang, W.; Average distance, radius and remoteness of a graph. Ars Math.\ Contemp.\ {\bf 7} no.\ 2 (2014), 441-452.

\bibitem{Zel1968} Zelinka, B.; Medians and peripherians of trees. Arch.\ Math.\ (Brno) {\bf 4} (1968), 87-95.
\end{thebibliography}
\end{document}